\documentclass[11pt]{amsart}
\usepackage{amsfonts}
\usepackage{amsmath}
\usepackage{enumerate}
\usepackage{color}
\usepackage{graphicx}
\usepackage{hyperref}
\usepackage[a4paper,top=3.5cm,bottom=3.5cm,left=3.5cm,right=3.5cm]{geometry}

\newtheorem{theorem}{Theorem}[section]
\newtheorem{corollary}[theorem]{Corollary}
\newtheorem{remark}[theorem]{Remark}
\newtheorem{lemma}[theorem]{Lemma}
\numberwithin{equation}{section}

\newcommand{\life}{t_*}
\newcommand{\jump}{P}
\newcommand{\SP}{X}
\newcommand{\B}{\mathcal{B}}
\newcommand{\EL}{\mathcal{L}}

\begin{document}

\title{A Model of Seasonal Savanna Dynamics}\thanks{This research was supported in part by the Polish NCN grant
2017/27/B/ST1/00100.}
 
 \author{
  Pawe\l{}~Klimasara}
  \address{Pawe\l{}~Klimasara, 
   Chair of Cognitive Science and Mathematical Modelling,
   University of Information Technology and Management in Rzesz\'ow,
   Sucharskiego 2, 35-225 Rzesz\'ow, Poland}
   \email{pklimasara@wsiz.edu.pl}
\author{
   Marta~Tyran-Kami\'nska}
   \address{ Marta~Tyran-Kami\'nska,
   Institute of Mathematics,
   University of Silesia in Katowice,
   Bankowa 14, 40-007 Katowice, Poland} 
   \email{marta.tyran-kaminska@us.edu.pl}

\begin{abstract}
    We introduce a mathematical model of savanna vegetation dynamics. The usual approach of nonequilibrium ecology is extended by including  the impact of wet and dry seasons. We present and rigorously analyze a model describing a mixed woodland-grassland ecosystem with stochastic environmental noise in the form of vegetation biomass losses manifesting fires. Both, the probability of ignition and the strength of these losses depend on the current season (as well as vegetation growth rates etc.). Formally it requires an introduction and analysis of a system that is a piecewise deterministic Markov process with parameters switching between given constant periods of time. We study the long time behavior of time averages for such processes.
\end{abstract}

\maketitle

\section{Introduction}
Seasonality is a very important feature of various ecological systems that affects their characterization in many ways. Defined as persistent periodic changes of environmental variables like temperature, rainfall, etc. it is crucial to understand population dynamics of many systems \cite{white2020seasonality}. Despite its importance and universality, seasonality is usually not explicitly present in mathematical modeling attempts in ecology. Existing formal inclusion of seasons in models is often analyzed only numerically or based on Floquet theory \cite{klausmeier2008floquet,white2020seasonality}. We propose a seasonal model that is formally a stochastic hybrid process that jumps between two piecewise deterministic Markov processes (PDMPs, \cite{davis84}) reflecting repeated switching between two seasons. Although we focus on the example of savanna dynamics model, we provide a general theory that can be used for other, formally similar, models or in situations with more than two seasons present.

Savannas are biomes characterized generally as mixed tree-grass systems \cite{scholesarcher1997tree} and cover around 20\% of Earth's land surface. The competition for resources between trees and grasses is regulated by many factors including herbivore activity, temporary changes in water availability and fires \cite{van2019effects}. There is a rich literature on savanna models \cite{yatat2018tribute} based on incorporating into dynamical system vegetation losses due to fires with constant \cite{hoyer-leitzel2021, yatatdumont2021minimalistic} or random \cite{d2006probabilistic,baudena2010idealized} frequency. Despite its ecological significance and prospective impact on model parameters, these approaches do not include explicit representation of seasonality. We take into account facts that in humid/mesic savannas rainfall happens primarily in wet seasons, boosting the vegetation growth, and results in more grass fuel for fires, happening more frequently in dry seasons, that cause then more damage to tree cover (see \cite{williams1999fire,n2018season,accatino2013humid,werner2012growth} and the references therein).
Most up-to-date savanna dynamics models that take  rainfall and/or soil moisture into account refer to their mean annual value (e.g. \cite{van2003effects,synodinos2018impact,staver2011tree}). Even when annual mean rainfall changes each year then these are much smaller variations in water availability than between seasons. Moreover,  the duration of wet and dry seasons usually are not the same. Nevertheless, there is no direct presence of wet and dry seasons in these models.

In Section \ref{s:model2d} we introduce a simple seasonal model of savanna vegetation dynamics. A system of logistic equations describes growth of tree and grass biomasses and without disturbances it would result in woodland (the trees outcompete grasses). We add random fire events manifested as discrete biomass losses. The probability of ignition and fire severity increase with grass biomass (fuel load). Later in Section \ref{s:model4d} we focus on more complicated version of this model where we introduce two more equations describing grazers and browsers populations that additionally impact the vegetation dynamics. We provide figures of sample trajectories illustrating the behavior of these systems.
The resulting models are stochastic only due to randomly occurring fires. The seasons are present in these models as repeated deterministic switching of growth rate parameters. This is entirely different setting than random switching between model parameters that has been used recently in PDMP models, e.g. in ecological dynamics \cite{benaim2016,benaim2018,cloez2017,hening2019,hening2020,hening2021,hening2021arxiv}, epidemiology \cite{benaim2019}, or population genetics \cite{guillin2019arxiv}.

To follow seasonal changes we introduce additional time variable measuring the duration of stay in a given season.   This allows us to   represent the savanna models as PDMPs in Section~\ref{s:pdmp} and provide sufficient conditions for their ergodicity (Theorem~\ref{th:main}).
Due to periodic changes we cannot study the usual convergence of distributions of such processes and we must look at convergence of time averages.
In Section~\ref{s:mean} we explore formally the long time behavior of averages of homogeneous Markov processes and we formulate  one of the main results of the paper that
$T$-processes, as in  \cite{tweedie79, meyn_tweedieII}, satisfying a~Foster-Lyapunov type condition (CD2) in~\cite{meyn_tweedieII} are mean-ergodic (Theorem \ref{t:mean}). Then we show that our savanna model PDMPs are such $T$-processes (Theorem~\ref{th:mainVT}) which implies
Theorem~\ref{th:main}. In Section~\ref{s:proof} we provide the proof of Theorem~\ref{t:mean}.
The paper concludes with a~short discussion.

\section{A  basic model of savanna dynamics with seasonality}\label{s:model2d}

 We start with adding seasonality into a simple model to grasp the actual problem with such modeling approach without intricacies of extended models rich in details and parameters. Basically as our minimal model we continue our work from \cite{klimasara2018model} based on \cite{beckage2011grass} and modify the model presented there.
It is a simple competition model between trees and grasses referred to as their biomass amounts (denoted as $W$ and $G$ respectively) in the system of differential equations:
  \begin{equation*}
\begin{cases}
\frac{dW}{dt}=r_w W\left(1-\frac{W}{K_w}\right),\\ \frac{dG}{dt}=r_g G\left(1-\frac{G}{K_g}-\frac{W}{K_w}\right),
\end{cases}
\end{equation*}
where  $r_w$ and $r_g$ are the respective growth rates, while the carrying capacities for the biomass amounts are $K_w$ and $K_g$.
We normalize both 'amount of biomass' variables to lie in $[0,1]$ by the change of variables:
\[
w(t)=\frac{W(t)}{K_w},\quad g(t)=\frac{G(t)}{K_g},
\]
and hence the model has the form:
\begin{equation}\label{no-fire}
\begin{cases}
\frac{dw}{dt}=r_w w\left(1-w\right), \\
\frac{dg}{dt}=r_g g\left(1-g-w\right).
\end{cases}
\end{equation}
Observe that \eqref{no-fire} has three stationary solutions $(1,0)$, $(0,0)$, and $(0,1)$, and that the point $(1,0)$ is asymptotically stable.

We add fires to this model and assume that they occur randomly with
\[
\Pr \big(\text{occurrence of fire in }(t, t+\Delta t)\,\big|\,w(t)=w, g(t)=g\big)=
\lambda(w,g)\Delta t+o(\Delta t),
\]
where  the function $\lambda\colon [0,1]^2\to \mathbb{R}_+$ is continuous.
We denote the consecutive moments of fire events by
 $t_1, t_2, \ldots$ The impact of fire in the model is implemented as the appropriate biomass losses according to
\begin{equation}\label{loss.eq}
\begin{cases}
w(t_n)=w(t_n^-)-M_w \, w(t_n^-),\\
g(t_n)=g(t_n^-)-M_g \, g(t_n^-),
\end{cases}
\end{equation}
where $ M_w,M_g\in (0,1)$ are constants and $v(t^-)=\lim_{s\rightarrow t^-}v(s)$ for $v\in\{w, g\}$. When fires occur at fixed deterministic times $t_{n+1}=t_{n}+\tau$, where $\tau$ is a constant,  one obtains impulsive systems  (see e.g. \cite{yatat2018tribute} or \cite{hoyer-leitzel2021} with $\alpha=1$).

The assumption that impact of fires is described discretely via constant biomass losses can be improved by a more general setting of random losses. To this end we replace the constants $M_w$ and $M_g$ with random variables. Their distribution can depend on the current biomass amounts.
Moreover such setup can be extended even more by including the seasonality. Thus we introduce two savanna seasons (wet and dry) and code them with  variable $i$, where $i=0$ refers to the dry season while $i=1$ to the wet one.  Some model parameters change between seasons. Thus e.g. $r_w^i$ and  $r_g^i$ denote the growth rates in the $i$th season. The seasons are time intervals changing alternately and to include this fact in the model we add a new clock variable $\zeta$ describing how long the current season lasts and hence schedules the moments when variable $i$ switches its value. The length of the $i$th season will be denoted by the constant value $\zeta_m^i$. Additionally, by introducing a two-dimensional variable $\xi$ for biomass amounts, the differential equation in the $i$th season takes the final form:
\begin{equation}\label{twoeq}
\begin{cases}
\frac{d\xi}{dt}=b^i(\xi),\\
\frac{d\zeta}{dt}=1,
\end{cases}
 \quad \text{where} \quad \xi=\begin{pmatrix}
w\\
g
\end{pmatrix} \quad \text{and} \quad
b^i(\xi)
= \begin{pmatrix}
r_w^i w\left(1-w\right)\\
r_g^i g\left(1-g-w\right)
\end{pmatrix}.
\end{equation}
Each time $\zeta$ reaches its maximal value $\zeta^i_m$, the present season ends and hence we reset the \emph{'duration of stay in a season'} that is the value of $\zeta$  to 0 and swap the model dynamics by changing all the affected parameters (via switching $i$ to $1-i$ everywhere). Note that the long time behavior of $\xi$ is the same as for \eqref{no-fire}.

Accordingly, the introduction of seasons changes the fire events description to:
\begin{multline}\label{prfi}
\Pr \big(\text{occurrence of fire in }(t, t+\Delta t)\,\big|\,\xi(t)=\xi, \zeta(t)=\zeta, i(t)=i\big)\\=
\lambda^i(\xi,\zeta)\Delta t+o(\Delta t),
\end{multline}
where $\lambda^i$ is a positive continuous function.
We assume that in the $i$th season for each $\xi$ and $\zeta$ there exists a probability measure $\mathcal{P}^i(\xi,\zeta, A)$ describing both biomass changes due to random fire events
\begin{equation}\label{amfi}
\Pr\big(\xi(t_n)\in A\,\big|\,\xi(t_n^{-})=\xi, \zeta(t_n^{-})=\zeta, i(t_n^{-})=i\big) = \mathcal{P}^i(\xi,\zeta, A)
\end{equation}
for any Borel subset $A$ of $\mathbb{R}^2$. In particular, we consider
\begin{equation}\label{d:amfi}
\mathcal{P}^i(\xi,\zeta, A)=\int_{\Theta}\mathbf{1}_A\big(S^i_\theta(\xi)\big)p_\theta^i(\xi,\zeta)\nu^i(d\theta),
\end{equation}
where $\Theta=(0,1)^2$, $\nu^i$ is a Borel measure on $\Theta$, $(\theta,\xi,\zeta)\to p_\theta^i(\xi,\zeta)$ is a continuous function such that
\begin{equation}\label{e:ptheta}
\int_{\Theta}p_\theta^i(\xi,\zeta)\nu^i(d\theta)=1.
\end{equation} The transformation $S_\theta^i$ describes the biomass loss due to fire and to simplify presentation we take
\begin{equation} \label{d:Stheta}
S_\theta^i(\xi)=\big((1-\theta_w)w,(1-\theta_g)g\big), \quad \xi=(w,g)\in (0,1)\times(0,1],\quad \theta=(\theta_w,\theta_g).
\end{equation}
Assuming that these losses are constant fractions of available amounts before the fire incident we have $p_\theta^i(\xi,\zeta)\equiv 1$ and  $\nu^i(d\theta)=\delta_{(M_w^i,M_g^i)}(d\theta)$, where $M_w^i,M_g^i\in (0,1)$ are constants and $\delta_M$ is the Dirac measure at the point $M=(M_w^i,M_g^i)$.
On the other hand when these losses are random we can take as $\nu^i$ the usual Lebesgue measure on the unit square $(0,1)^2$. Then for each $(\xi,\zeta)$ the function $\theta\mapsto p_\theta^i(\xi,\zeta)$ describes the density of the distribution of biomass losses due to fire. In Figure \ref{fig:2d} we display  sample graphs of wood and grass biomasses in time, including losses due to random fires and changes of seasons.
\begin{figure}[htb]
    \centering
    \includegraphics[width=0.75\textwidth]{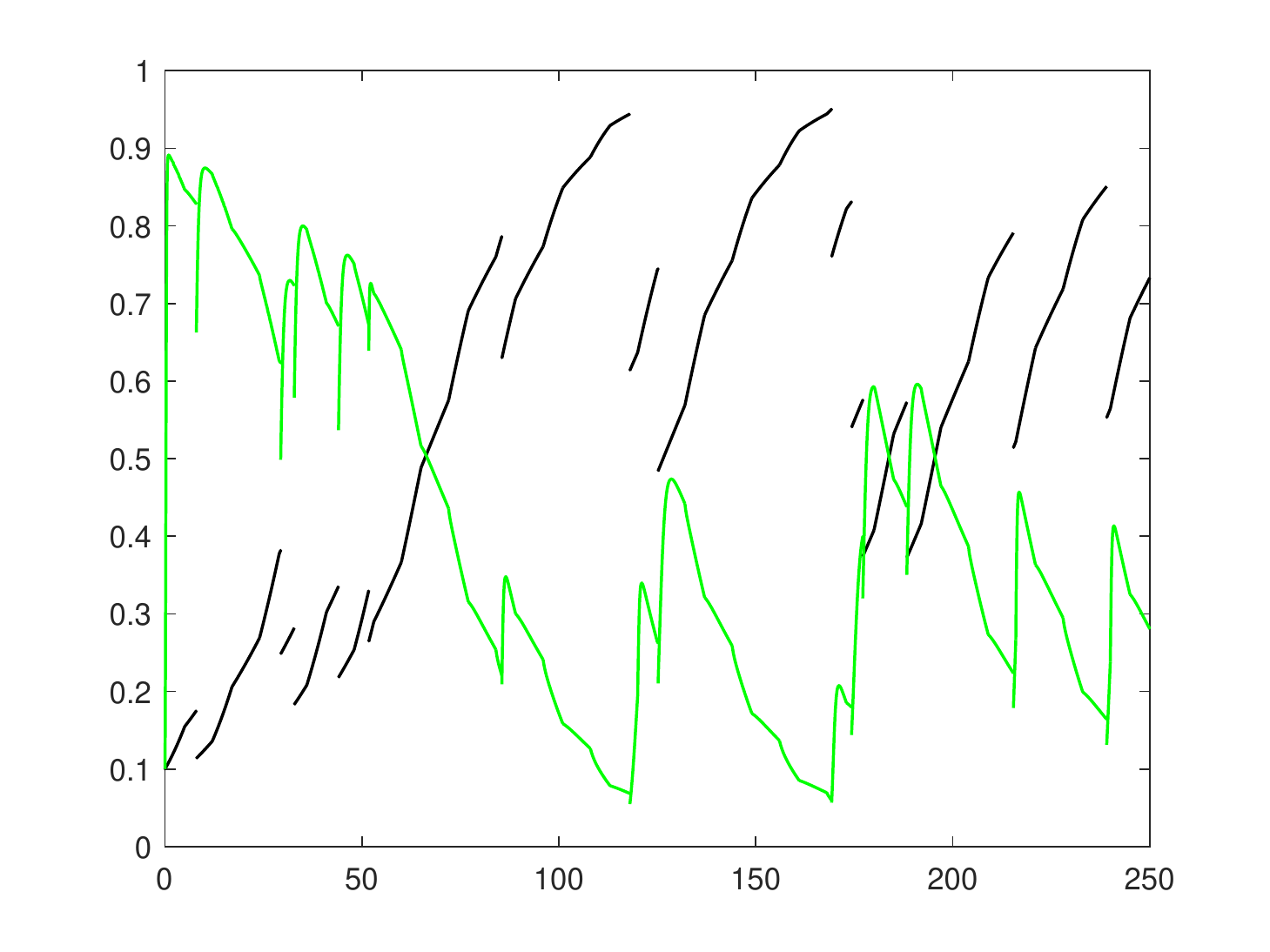}
    \caption{Sample trajectories of the stochastic process in \eqref{twoeq}--\eqref{amfi} with parameters for the dry season $r_w^0=0.05$, $r_g^0= 2.5$, $M_w^0= 0.35$, $M_g^0= 0.2$, $\lambda^0(w,g,\zeta)=0.09g+ 0.01$, $\zeta^0_m=7$ and for the wet season $r_w^1=0.1$, $r_g^1=10.75$, $M_w^1=0.2$, $M_g^1= 0.05$, $\lambda^1(w,g,\zeta)= 0.001g+0.02$, $
\zeta^1_m=5$. The green line represents the graph of the grass biomass amount over time $t\mapsto g(t)$, and the black line refers to the wood biomass $t\mapsto w(t)$}
    \label{fig:2d}
\end{figure}

\section{A savanna model featuring herbivores and seasonality}\label{s:model4d}

We extend the model from the previous section  by adding populations of herbivores depending on the food availability (grass for grazers and trees for browsers). We start  with introduction of the population dynamics model that we later complete by adding random fire events and seasonality. The differential equations describing the dynamics of tree and grass biomasses contain additional terms referring to the presence of herbivores:
\begin{equation*}%
\begin{cases}
\frac{dW}{dt}=r_w W\left(1-\frac{W}{K_w}\right)-c_W H_B W,\\  \frac{dG}{dt}=r_g G\left(1-\frac{G}{K_g}-\frac{W}{K_w}\right)-c_G H_G G,
\end{cases}
\end{equation*}
where $H_G$, $H_B$  are populations of grazers and browsers and $c_W$, $c_G$ denote consumption coefficients of woody/grass biomass by browsers/grazers, accordingly.  We describe the population dynamics of herbivores as in \cite{van2019effects} by:
  \begin{equation*}%
\begin{cases}
\frac{dH_G}{dt}=e_G H_G G-d_G H_G^2,\\
\frac{dH_B}{dt}=e_W H_B W-d_B H_B^2,
\end{cases}
\end{equation*}
where $e_W$, $e_G$  are  consumption and conversion efficiency coefficients of woody/grass biomass by browsers/grazers
and $d_B$, $d_G$ denote death rates of browsers and grazers, respectively.

Similarly to the model from Section \ref{s:model2d} we normalize biomass amounts and additionally redefine the herbivore population variables by
\[
w(t)=\frac{W(t)}{K_w},\quad g(t)=\frac{G(t)}{K_g},
\quad h_G(t)=\frac{d_G H_G(t)}{e_G K_g}, \quad h_B(t)=\frac{d_B H_B(t)}{e_W K_w},
\]
which enforces us to change the parameters as well
\[
c_w\equiv c_W \frac{e_g}{d_G},\quad c_g\equiv c_G \frac{e_w}{d_B}, \quad e_w\equiv e_WK_W,\quad e_g\equiv e_GK_G.
\]
These modifications lead to the simpler system of differential equations:
\begin{equation}\label{no-fire.herb}
\begin{cases}
\frac{dw}{dt}=r_w w\left(1-w\right)-c_w h_B w,\\
\frac{dg}{dt}=r_g g\left(1-g-w\right)-c_g h_G g,\\
\frac{dh_G}{dt}=e_g h_G\left(g-h_G\right),\\
\frac{dh_B}{dt}=e_w h_B\left(w-h_B\right).
\end{cases}
\end{equation}
This system has a unique positive stationary point
\[
w= \frac{r_w}{r_w+c_w}, \quad g= \frac{r_g}{r_g+c_g}\frac{c_w}{r_w+c_w}, \quad h_G=g, \quad h_B=w,
\]
and it is asymptotically stable.
Again, we add alternating seasons, dry ($i=0$) and wet ($i=1$), by changing the plant growth rates $r^i_w$, $r^i_g$ along with them. We illustrate the long time behavior of this system in Figure \ref{fig:herb}. A typical periodicity of seasonal models is clearly visible in this figure.
\begin{figure}[htb]
    \centering
    \includegraphics[width=0.75\textwidth]{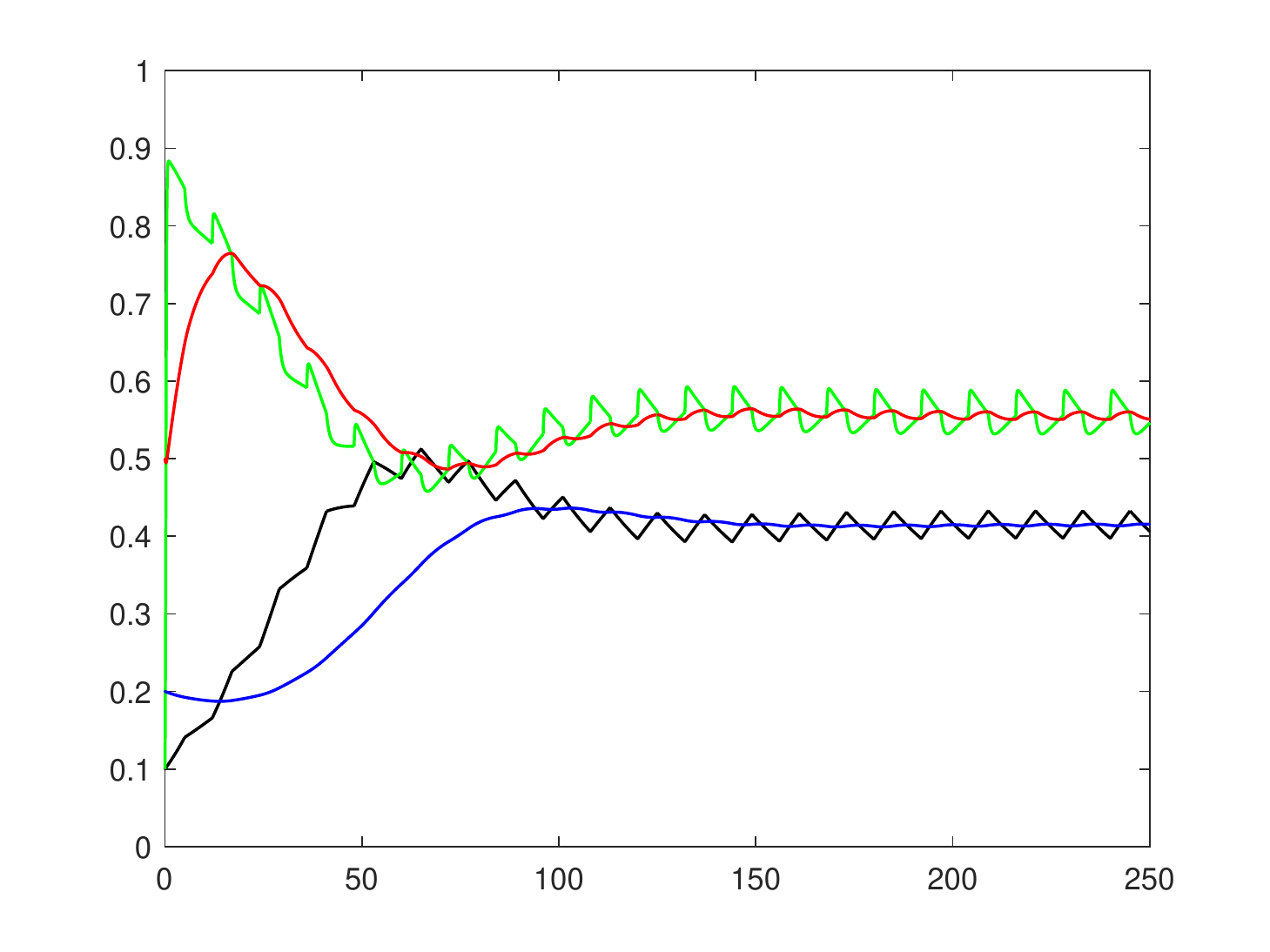}
    \caption{Deterministic trajectories for system \eqref{no-fire.herb} with alternating seasons and initial condition $w=g=0.1$, $h_G=0.5$, $h_B=0.2$. We used the same color references and parameters as in Figure \ref{fig:2d} and additionally $c_w=e_w=0.1$, $c_g=e_g=0.2$. The red line represents the graph of the population of grazers over time $t\mapsto h_G(t)$ while the blue line refers to the population of browsers $t\mapsto h_B(t)$}
    \label{fig:herb}
\end{figure}

Finally we may incorporate the fire events into this model in analogy to the basic no-herbivore model.
Now we have a 4-dimensional vector $\xi=(w,g,h_G,h_B)$ and the dynamics is given by equations
\eqref{twoeq} with the values for $b^i(\xi)$ taken from system \eqref{no-fire.herb}. Fire-related probabilities, \eqref{prfi} and \eqref{amfi}, remain unchanged, while the transformation $S_\theta^i$ takes the form
\begin{equation} \label{d:Stheta4d}
S_\theta^i(\xi)=\big((1-\theta_w)w,(1-\theta_g)g, h_G,h_B\big),\quad  \xi=(w,g,h_G,h_B),\quad \theta=(\theta_w,\theta_g).
\end{equation}
A sample trajectory of the main model containing all the stochastic effects is presented in Figure \ref{fig:herbr}.
\begin{figure}[htb]
    \centering
    \includegraphics[width=0.75\textwidth]{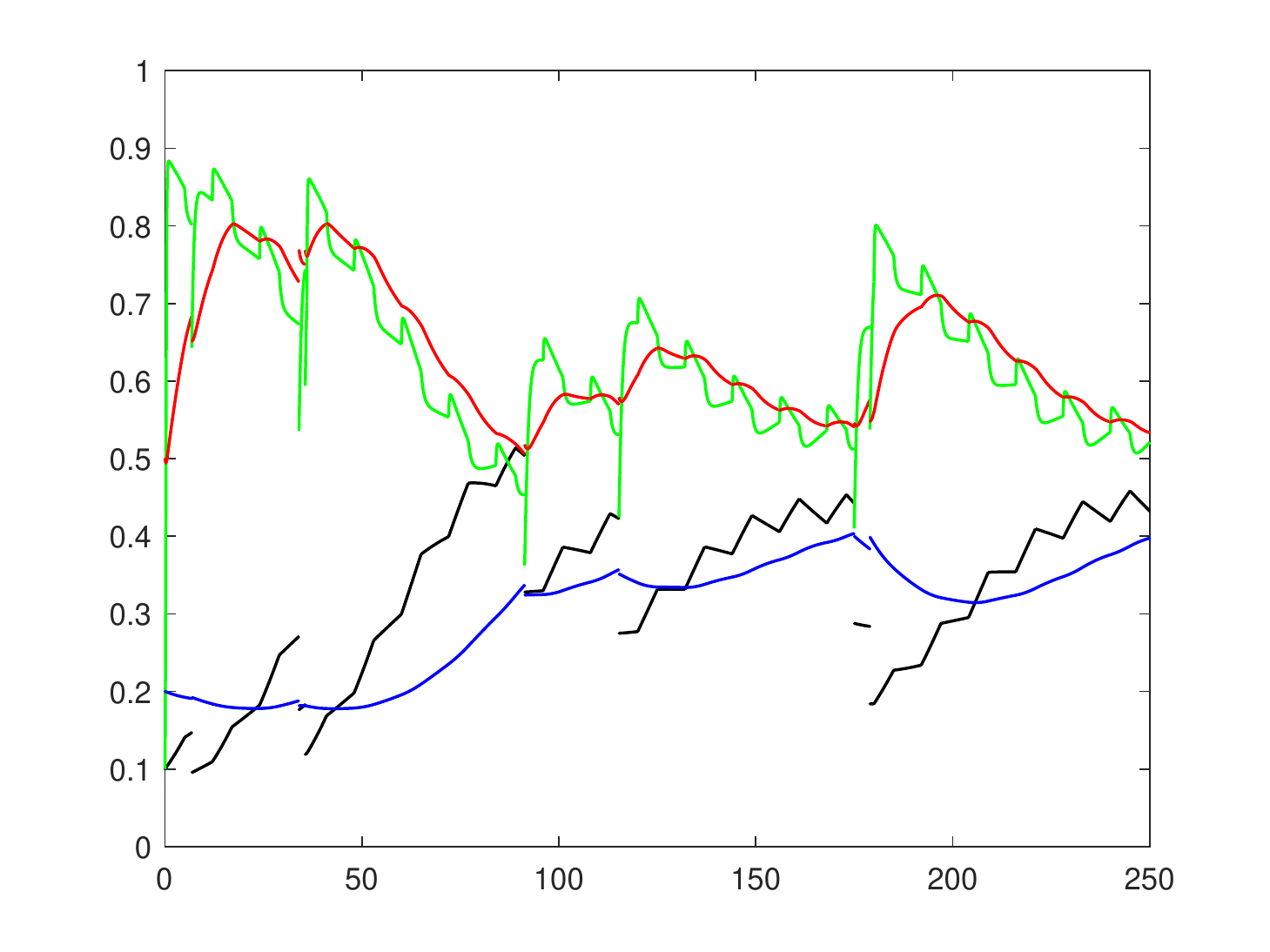}
    \caption{Sample trajectories for the stochastic model of savanna vegetation dynamics with herbivores, random fires and seasonality. The parameters and colors are the same as in Figure \ref{fig:herb}}
    \label{fig:herbr}
\end{figure}

\section{PDMPs and seasonality}\label{s:pdmp}

In this section we recognize introduced savanna models as PDMPs with the aim to show that such processes can be used to study seasonality in ecological/population models. After brief introduction of the theory basics we formulate one of the main results of this paper concerning the long term behavior of savanna models.
For general background on PDMPs we refer the reader to \cite{davis93, rudnickityran17}.

We consider two flows that arise as solutions of ordinary differential equations
\begin{equation}\label{e:de}
\xi'(t)=b^i\big(\xi(t)\big),
\end{equation}
where $b^i\colon \mathbb{R}^d\to\mathbb{R}^d$ is a (locally) Lipschitz continuous mapping.
We assume that $\SP_i$ is a  Borel subset of $\mathbb{R}^d$ such that for each $\xi_0\in \SP_i$ the solution $\xi(t)$ of
\eqref{e:de} with initial condition $\xi(0)=\xi_0$ exists and $\xi(t)\in \SP_i$ for all $t\ge 0$. We denote this solution by
$\varphi_t^i(\xi_0)$, $i=0,1$.
We also introduce the clock variable $\zeta$ and the season variable $i$. Thus, the variable $x=(\xi,\zeta,i)$  changes in time according to the flow
\begin{equation}\label{d:phi}
\phi_t(x)=\phi_t(\xi,\zeta,i)=\big(\varphi_t^i(\xi),\zeta+t,i\big) .
\end{equation}

If we consider the 2-dimensional model from Section \ref{s:model2d} (no herbivores) then equations \eqref{e:de} and \eqref{d:phi} introducing the flow $\phi_t$ correspond to equation \eqref{twoeq} with $\xi=(w,g)\in \SP_i$, where $\SP_i=(0,1)\times(0,1]$ and $i=0,1$, while for the 4-dimensional model from Section \ref{s:model4d} (with grazers and browsers) we have $\xi=(w,g,h_G,h_B)\in \SP_i$ with $\SP_i=(0,1)\times(0,1]\times(0,\infty)^2$.

Our state space is
\[
\SP=\bigcup_{i}\SP_i\times[0,\zeta^i_m)\times
\{i\},\]
where $\zeta_m^i$ is the length of the $i$th season. The flow $\{\phi_t\}$ can exit
the set $\SP$ in a finite positive time  through a boundary $\Gamma$ of $\SP$.  Under our assumptions  we have
\[
\Gamma=\bigcup_{i}\SP_i\times \{\zeta_m^i\}\times \{i\}
\]
and the \emph{hitting time} of the boundary $\Gamma$ is given by
\begin{equation}\label{d:exit}
\life(x)=\inf\{t>0:\phi_{ t}(x)\in \Gamma \}= \zeta_m^i-\zeta \quad \text{for } x=(\xi,\zeta,i) \in \SP.
\end{equation}
If the state of the process at the end of a given season is represented by the point $(\xi,\zeta_m^i,i)$ from the boundary $\Gamma$, then the process moves to the point $(\xi,0,1-i)$ at the beginning of the next season.
Thus, jumps are described
by a stochastic kernel
$\jump $ defined by
\[
\jump(x,B)=\int_{\Theta}
\mathbf{1}_B \big(\mathbf{S}(x,\theta)\big) \nu(x,d\theta), \quad x\in \SP\cup \Gamma, B\in \B,
\]
where $\mathbf{S}\colon (\SP\cup \Gamma)\times \Theta \to \SP $ is a measurable transformation and $\nu(x,\cdot)$ is a stochastic kernel. In reference to \eqref{d:amfi}, we consider
\begin{equation}\label{d:S}
\mathbf{S}(x,\theta)=\mathbf{S}(\xi,\zeta,i,\theta)=\begin{cases}
\big(S_\theta^i(\xi),\zeta,i\big),&\text{if } \zeta <\zeta_m^i, \\
(\xi,0,1-i), & \text{if } \zeta=\zeta_m^i,
\end{cases}
\end{equation}
and
\begin{equation}\label{d:nu}
\nu(x,d\theta )=\begin{cases}
 p_\theta^i(\xi,\zeta)\nu^i(d\theta),   &  \text{if } \zeta <\zeta_m^i,\\
  \nu^i(d\theta),   &  \text{if }  \zeta=\zeta_m^i.
\end{cases}
\end{equation}
 Finally, let the jump rate function be defined by
$q(\xi,\zeta,i)=\lambda^i(\xi,\zeta)$ for  $(\xi,\zeta,i)\in \SP$.
For each $x\in \SP$ we define
\begin{equation}\label{d:tail}
F_x(t)=\mathbf{1}_{[0,\life(x))}(t)\exp\left\{-\int_0^t q\big(\phi_r(x)\big)dr\right\},  \quad t\ge 0,
\end{equation}
where $\phi$ is as in \eqref{d:phi}.
If we start at the point $\Psi_0=(\xi_0,\zeta_0,i_0)$ at time $\tau_0$, then we follow the path  $t\mapsto \phi_{t-\tau_0}(\Psi_0)$ up to the occurrence of either the fire or the next season, whichever comes first. Thus the next jump time $\tau_1$ is chosen according to the distribution
\[
\mathbb{P}(\tau_1-\tau_0>t\,|\,\Psi_0=x)=F_x(t).
\]
Then we define
\[
\Phi(t)=\phi_{t-\tau_0}(\Psi_0), \quad \Phi_1=\phi_{\tau_1-\tau_0}(\Psi_0), \quad \Psi_1=\mathbf{S}(\Phi_1,\vartheta_1), \]
where $\vartheta_1$ is a random variable with distribution $\nu(\Phi_1,\cdot)$,
and we restart the process from the point $\Psi_1$. In this way we define a sequence  $\Psi_n$ of $\SP$-valued random variables and jump times $\tau_n$ such that the process $\Phi=\{\Phi(t):t\ge 0\}$ is defined by
\begin{equation}\label{d:Phi}
\Phi(t)=\phi_{t-\tau_n}(\Psi_n) \quad \text{for }\tau_n\le t<\tau_{n+1}, \end{equation}
where
\begin{equation}\label{d:Psin}
\Psi_n=\mathbf{S}\big(\phi_{\sigma_n}(\Psi_n),
\vartheta_n\big), \quad \sigma_n=\tau_n-\tau_{n-1},
\end{equation}
and
$\vartheta_n$ is a  $\Theta$-valued random variable with distribution $\nu(\phi_{\sigma_n}(\Psi_n),\cdot)$, $n\in \mathbb{N}$.

We conclude the section with  the main theorem of this paper concerning each of the Markov processes $\Phi=\{\Phi(t):t\ge 0\}$ representing the models from Sections \ref{s:model2d} and~\ref{s:model4d}. Let $\mathbb{P}_x$ denote the law of the process $\Phi$ with initial condition $\Phi(0)=x$, $x\in \SP$.

We assume that
the functions  $\lambda^i$ and $p_\theta^i$ satisfy the following:
\begin{enumerate}[(i)]
    \item\label{A1} their values depend only on $w$, $g$, and $\zeta$ in each case (there is no direct influence of herbivores on fire ignition nor severity),
    \item\label{A2} $\lambda^i$ is strictly positive in each season (fires should be always possible but of course much more probable during the dry season),
    \item\label{A3} there are $a_w, a_g\in (0,1]$  and $\varepsilon_w,\varepsilon_g>0$ such that
    \begin{equation}\label{c:lam}
   \lambda^i(w,g,\zeta) \int_{\Theta}\left[ \frac{1}{(1-\theta_w)^{a_w}}  -1\right] p_\theta^i(w,g,\zeta)\nu_i(d\theta)  -a_w r_w^i\le -\varepsilon_w
    \end{equation}
    for all $\zeta\in [0,\zeta_m^i)$, $g\in (0,1]$ and $w$ from a neighbourhood of $0$,
    and
    \begin{multline}\label{c:lamg}
    \lambda^i(w,g,\zeta) \int_{\Theta}\left[ \frac{1}{(1-\theta_g)^{a_g}}  -1 +g^{a_g}\ln\frac{1-w}{1-(1-\theta_w)w}\right] p_\theta^i(w,g,\zeta)\nu_i(d\theta)\\
    -a_g r_g^i(1-w)\le -\varepsilon_g
    \end{multline}
    for all $\zeta\in [0,\zeta_m^i)$, $w\in (0,1)$ and $g$ from a neighbourhood of $0$,
    \item\label{A4} for $a=(a_w,a_g)$ as in \eqref{A3} we have
\begin{equation*}
\int_0^1\left[\frac{1}{(1-\theta_w)^{a_w}}+\frac{1}{(1-\theta_g)^{a_g}}-\ln(1-(1-\theta_w)w)\right]p_\theta^i(w,g,\zeta)\nu_i(d\theta)<\infty
\end{equation*}
for all $(w,g)\in (0,1)\times(0,1]$, $\zeta\in [0,\zeta_m^i).$
\end{enumerate}
Conditions  \eqref{A3}–\eqref{A4} are technical assumptions allowing a construction of a Lyapunov function controlling survival of woods and grasses (the behaviour of the process when $w$ or $g$ are close to zero). In particular, conditions \eqref{c:lam} and \eqref{c:lamg}  prevent the total loss of wood and grass biomasses, respectively.

\begin{theorem}\label{th:main} Suppose that \eqref{A1}–\eqref{A4} hold. Then for each  $x=(\xi,\zeta,i)\in \SP$ there exists a probability measure $\Pi(x,\cdot)$ on $\SP$ such that
\[
\lim_{t\to \infty}\frac{1}{t}\int_0^t \mathbb{P}_x(\Phi(s)\in B)ds=\Pi(x,B), \quad \text{for all }B\in \B,
\]
and for any bounded Borel measurable $f$ we have
\[
\mathbb{P}_x\left(\lim_{t\to\infty}\frac{1}{t}\int_0^t f(\Phi(s))ds=\int f d\tilde\Pi\right)=1
\]
for a random measure $\tilde\Pi$ satisfying $\Pi(x,B)=\mathbb{E}_x\tilde\Pi(B)$, $B\in \B$, $x\in \SP$.
\end{theorem}
The proof of Theorem~\ref{th:main} will be given in the next section. In fact we will show that the convergence in  Theorem~\ref{th:main} is uniform with respect to all sets $B$ and that our savanna models are $T$-processes satisfying a  Foster–Lyapunov type condition (see Theorem~\ref{th:mainVT}).

We finish the section with the conclusion regarding the model from \cite{klimasara2018model} extended by inclusion of seasonality and (possibly) herbivore activity.
\begin{corollary}\label{cor} Suppose that  the losses are constant fractions  ($M_w^i, M_g^i$) of the tree/grass biomass and that $\lambda^i(w,g,\zeta)=\lambda_0^i g$ with $\lambda_0^i>0$, $i=0,1$. If
\begin{equation}\label{c:inra}
r_w^i  +\lambda_0^i \ln(1-M_w^i)  >0, \quad i=0,1,
     \end{equation}
then  Theorem \ref{th:main}  Condition holds.
\end{corollary}
\begin{proof}  From condition \eqref{c:inra} it follows that
there exists $a_w\in (0,1]$ such that
   \[
   \lambda_0^i \left[ \frac{1}{(1-M_w^i)^{a_w}}  -1\right]  -a_w r_w^i<0, \quad i=0,1,
    \]
implying condition \eqref{c:lam}.
Now observe that the left-hand side of \eqref{c:lamg} is of the form
\[
    \lambda_0^ig \left[ \frac{1}{(1-M_g^i)^{a_g}}  -1 +g^{a_g}\ln\frac{1-w}{1-(1-M_w^i)w}\right] -a_g r_g^i(1-w)
\]
and, for $w\in(0,1)$ and $g$   from a neighbourhood of $0$, it is always negative. Consequently,  assumptions  \eqref{A1}–\eqref{A4} are satisfied.
\end{proof}

\begin{remark}
In the simplest model  as in Corollary \ref{cor} note that condition  \eqref{c:inra}   implies that $r_w^i  +\lambda_0^i g \ln(1-M_w^i)  >0$ for all $g\in(0,1]$, $i=0,1$. Thus the mean growth rate of wood biomass is positive in the limit $w\to 0$ in both seasons allowing wood-grass coexistence (in the presence of random fires). \end{remark}
\section{Mean ergodic Markov processes}\label{s:mean}

Following 
\cite{meyn_tweedieI,meyn_tweedieII,meyn_tweedieIII}, we summarize briefly necessary concepts to study the long time behavior of Markov processes.
Let $\SP$ be a locally compact separable metric space and let $\B$ denote the Borel subsets of $\SP$.
A function $T\colon \SP\times \B\to [0,1]$ is called a \emph{(substochastic) kernel} on $\SP$ if for $B\in \B$ the function $T(\cdot,B)$ is measurable and $T(x,\cdot)$ is a measure on $\B$ (satisfying $T(x,\SP)\le 1$ for each $x\in \SP$). The kernel is called \emph{non-trivial} if $T(x,\SP)>0$ for all $x\in \SP$ and \emph{stochastic} if $T(x,\SP)=1$ for all $x$.
A substochastic kernel $T$ defines a linear operator on the space of finite signed measures $\mathcal{M}(\SP)$ on $\B$.
For $\mu\in\mathcal{M}(\SP)$ we define a new signed measure $\mu T$ by
\[
\mu T(B)=\int_\SP T(x,B)\mu(dx).
\]
If $K$ and $T$ are two kernels their product $KT$ is defined as
\[
KT(x,B)=\int_{\SP}T(y,B)K(x,dy),\quad x\in \SP, B\in \B.
\]
A kernel $T$ is called a \emph{continuous component} of a  kernel $K$ on $\SP$ if it satisfies \linebreak $K(x,B)\ge T(x,B)$ for all $x\in \SP$, $B\in \B$ and the function $T(\cdot,B)$ is \emph{lower semicontinuous}, i.e.
\[
\liminf_{y\to x}T(y,B)\ge T(x,B), \quad x\in \SP.
\]

Let $\Phi=\{\Phi(t):t\ge 0\}$ be a continuous-time Markov process with state space $\SP$ and let $\mathbb{P}_x$ denote the law of the process $\Phi$ with initial condition $\Phi(0)=x$, $x\in \SP$.  We assume that $\Phi$ is strong Markov and has right-continuous sample paths with left limits. For each $t\ge 0$ the transition probability of the process is
\[
P^t(x,B)=\mathbb{P}_x\big(\Phi(t)\in B\big), \quad x\in \SP,B\in \B,
\]
and if the process is non-explosive then $P^t$ is a stochastic kernel.
Recall that the process $\Phi$ is \emph{non-explosive} if  there is an increasing sequence of open precompact sets $O_n$ such that $\SP=\bigcup_{n}O_n$ and for each $x\in \SP$ we have
\[
\mathbb{P}_x\big(\lim_{n\to \infty}\inf\{t\ge 0:\Phi(t)\not\in O_n\}=\infty\big)=1.
\]

An operator  $\EL$ is called the \emph{extended generator} of the Markov process $\Phi$   (see \cite{davis93}), if
its domain  $\mathcal{D}(\EL)$ consists of those measurable $V\colon \SP\to \mathbb{R}$ for which there exists a measurable $W\colon \SP\to \mathbb{R}$ such that the function $t\mapsto W\big(\Phi(t)\big)$ is integrable $\mathbb{P}_x$-a.s. for each $x\in \SP$ with the process
\[
t\mapsto V\big(\Phi(t)\big)-V(x)-\int_{0}^t W\big(\Phi(s)\big)\,ds
\]
being a $\mathbb{P}_x$-local martingale  and we define $\mathcal{L}V=W$.
A function $V\colon \SP\to [0,\infty]$ is said to be \emph{norm-like} if the sets $\{x\in \SP:V(x)\le r\}$ are precompact for all sufficiently large $r>0$.
It follows from \cite[Theorem 2.1]{meyn_tweedieIII} that if there exists a norm-like function $V\in \mathcal{D}(\EL)$ and constants $c,d\ge 0$ such that
\begin{equation}
\EL V(x)\le cV(x)+d, \quad x\in \SP
\end{equation}
then the process $\Phi$ is non-explosive.

For any $\mu\in \mathcal{M}(\SP)$ we define
the norm
\[
\|\mu\|=\sup_{B\in \B}|\mu(B)|,\quad \mu\in \mathcal{M}(\SP).
\]
It is equivalent to the total variation norm since we
have
$
\|\mu\|\le\|\mu\|_{TV}\le 2\|\mu\|.
$
The process $\Phi$ is called  \emph{Ces\'aro-ergodic} (or \emph{mean ergodic})
if for each probability measure $\mu$ there exists a measure $\mu\Pi \in \mathcal{M}(\SP)$ such that \begin{equation}\label{c:meg}
\lim_{t\to\infty}\left\|\frac{1}{t}\int_0^t \mu P^s(\cdot)ds -\mu \Pi\right\|=0.
\end{equation}
In that case we define
\[
\Pi(x,B)=\delta_x\Pi(B), \quad B\in \B, x\in \SP,
\]
where $\delta_x$ is the Dirac delta.
Recall that a probability measure $\pi$ is called \emph{invariant for the process} $\Phi$ if $\pi=\pi P^t$ for all $t$. In particular, each limiting measure $\mu\Pi$ in \eqref{c:meg} is invariant for the process $\Phi$.
Finally, the process $\Phi$ is called a \emph{$T$-process} if for some probability measure $a$ on $\mathbb{R}_+$ the kernel $K_a$ defined by
\begin{equation}\label{d:ka}
K_a(x,B)=\int_{0}^\infty P^t(x,B)a(dt).
\end{equation}
has a non-trivial continuous component.

We now impose a Foster--Lyapunov type condition corresponding to condition (CD2) in  \cite{meyn_tweedieIII}:
\begin{enumerate}[(V)]
\item\label{V} There exist a non-negative norm-like  $V\in \mathcal{D}(\EL)$, a measurable $f\colon \SP\to [1,\infty)$, a~compact set $C$ and positive constants $c,d$ such that
\begin{equation}\label{e:F-L}
\EL V(x)\le -c f(x)+d \mathbf{1}_{C}(x),\quad x\in \SP.
\end{equation}
\end{enumerate}

\begin{theorem}\label{t:mean}
Suppose that condition (V) holds and that the process $\Phi$ is a~$T$-process. Then $\Phi$
is mean ergodic and we have
\[
\mathbb{P}_x\left(\lim_{t\to\infty}\frac{1}{t}\int_0^t f(\Phi(s))ds=\int f d\tilde\Pi\right)=1
\]
for any bounded Borel measurable $f$
 and
for a random measure $\tilde\Pi$ satisfying  $\Pi(x,B)=\mathbb{E}_x\tilde\Pi(B)$, $B\in \B$, $x\in \SP$.
\end{theorem}
The proof of Theorem~\ref{t:mean} is given in Section~\ref{s:proof}.  We have the following direct consequence of Theorem~\ref{t:mean}.

\begin{corollary}
Suppose that condition (V) holds and that the process $\Phi$ is a~$T$-process with a unique invariant probability measure $\pi$. Then \begin{equation*}
\lim_{t\to\infty}\sup_{B\in \B}\left|\frac{1}{t}\int_0^t P^s(x,B)ds -\pi(B)\right|=0
\end{equation*}
and
\[
\mathbb{P}_x\left(\lim_{t\to\infty}\frac{1}{t}\int_0^t f(\Phi(s))ds=\int f d\pi\right)=1
\]
for all  $x\in \SP$  and all bounded Borel measurable $f$.
\end{corollary}



Our next result, along with Theorem \ref{t:mean}, implies
Theorem~\ref{th:main}
and shows that savanna models from Sections \ref{s:model2d} and \ref{s:model4d} are mean ergodic.

\begin{theorem}\label{th:mainVT}
Under assumptions \eqref{A1}--\eqref{A4} the Markov processes from Sections \ref{s:model2d} and \ref{s:model4d} satisfy condition (V) and  are $T$-processes.
\end{theorem}
\begin{proof}
We start by showing how condition (V) can be checked for our PDMP models.
Let $M(\SP)$ be the set of all measurable real-valued functions on $\SP$.  We define as in \cite{davis93} \[
M_\Gamma(\SP)=\{V\in M(\SP): V(x)=\lim_{t\downarrow 0}V\big(\phi_{-t}(x)\big) \text{ for } x\in \Gamma\}.\]
It can be shown as in the proof of \cite[Theorem 26.14]{davis93} and \cite[Theorem 18]{jacod} that the domain $\mathcal{D}(\EL)$ of the extended generator $\EL$ contains those functions $V\in M_\Gamma(\SP)$ that satisfy the following:
\begin{enumerate}
\item  the function $t\mapsto V\big(\phi_t(x)\big)$  is absolutely continuous on $[0,\life(x))$ for $x\in \SP$,
\item $V$ satisfies the boundary condition
\[
V(x)=\int_\SP V(y)\jump(x,dy), \quad x\in \Gamma,
\]
\item for each $x\in \SP$  and $t<\life(x)$
\begin{equation*}
\int_{0}^{t} \int_{\SP} \big|V(y)-V\big(\phi_s(x)\big)\big|P\big(\phi_s(x),dy\big)q\big(\phi_s(x)\big)ds <\infty.
\end{equation*}
\end{enumerate}

The formula for the extended generator $\EL$ is
\begin{equation*}
\EL V(x)=\EL_0V(x)+q(x)\int_\SP\big( V(y)-V(x)\big)P(x,dy),
\end{equation*}
where
\[
\EL_0V(x)=\lim_{t\downarrow 0}\frac{V\big(\phi_t(x)\big)-V(x)}{t}. \]

For $V\in \mathcal{D}(\EL)$ that is  a smooth function  of variables $\xi$ and $\zeta$ we have
\[
\EL V(\xi,\zeta,i)=\EL_0V(\xi,\zeta,i)+\lambda^i(\xi,\zeta)\int_{\Theta} \big(V\big(S^i_\theta(\xi),\zeta,i\big)-V(\xi,\zeta,i)\big)p_{\theta}^i(\xi,\zeta)\nu^i(d\theta),
\]
where
\begin{equation*}
\EL_0V(\xi,\zeta,i)=\sum_{j=1}^db_j^i(\xi)\frac{\partial V}{\partial \xi_j}(\xi,\zeta,i)+\frac{\partial V}{\partial \zeta}(\xi,\zeta,i),\quad \xi\in \SP_i,\zeta\in [0,\zeta_m^i), i=0,1,
\end{equation*}
and
the boundary condition is of the form
\begin{equation*}
V(\xi,\zeta_m^i,i)=V(\xi,0,1-i),\quad \xi\in \SP_i, i=0,1.
\end{equation*}
For $d=2$ and $\xi=(w,g)$ we take
\[
V_1(w,g,\zeta,i)=\frac{1}{w^{a_w}}+\frac{1}{g^{a_g}}-\ln(1-w)+\zeta\sqrt{\zeta_m^i-\zeta},
\]
while for $d=4$ and $\xi=(w,g,h_G,h_B)$ we consider
\[
V_2(w,g,h_G,h_B,\zeta,i)=V_1(w,g,\zeta,i)+\frac{1}{h_G} +\ln(1+h_G)+\frac{1}{h_B} +\ln(1+h_B).
\]
It is easily seen that both functions are in the domain of the corresponding extended generator.
Note that for $V=V_1$ and $V=V_2$ we have

\begin{align*}
V(S_\theta^i(\xi),\zeta,i)-V(\xi,\zeta,i)&=\frac{1}{w^{a_w}} \left[ \frac{1}{(1-\theta_w)^{a_w}}  -1\right]+\frac{1}{g^{a_g}} \left[ \frac{1}{(1-\theta_g)^{a_g}}  -1\right]\\
&\quad +\ln (1-w)-\ln(1-(1-\theta_w)w).
\end{align*}
Thus condition (V) holds, since $\EL V(\xi,\zeta,i)\to -\infty$ when $\xi$ tends to the boundary of $\SP_i$ or $\zeta\to \zeta_m^i$, by assumptions \eqref{A1} and \eqref{A3}.

Now we prove that the process $\Phi=\{\Phi(t):t\ge 0\}$ as in \eqref{d:Phi} is a  $T$-process. Since its probability transition function is given by
\begin{align*}
P^t(x,B)&=\mathbb{P}_x\big(\Phi(t)\in B\big)=\sum_{n=0}^\infty \mathbb{P}_x\big(\Phi(t)\in B, \tau_n\le t< \tau_{n+1}\big)\\
&=\sum_{n=0}^\infty \mathbb{P}_(\phi_{t-\tau_n}(\Psi_n)\in B, \tau_n\le t< \tau_{n+1}\big)
\end{align*}
for $x\in \SP$, $B\in \B$, it is enough to show that for each $x_0\in \SP$ there exist a constant $c_{x_0}>0$,  an open set $U_{x_0}$ containing $x_0$ and an open set $V_{x_0}$ such that
\begin{equation}\label{d:lower}
\int_0^\infty P^t(x,B)e^{-t}dt\ge  c_{x_0}\mathbf{1}_{U_{x_0}}(x)m(B\cap V_{x_0}), \quad B\in \B, x\in \SP,
\end{equation}
where $m$ is the product of the $(d+1)$-dimensional Lebesgue measure and the counting measure on $\{0,1\}$. The kernel $T_{x_0}(x,B)=c_{x_0}\mathbf{1}_{U_{x_0}}(x)m(B\cap V_{x_0})$ is a continuous component non-trivial at $x_0$ for $K_a$ with $a$ being the exponential distribution on $\mathbb{R}_+$. By taking a sequence of points $(x_k)$ such that  $\SP=\bigcup_{k}U_{x_k}$ we can define the kernel $T=\sum_{k=1}^\infty 2^{-k}T_{x_k}$ and conclude that $T$ is a continuous component non-trivial at every $x\in \SP$. It implies that $\Phi$ is a $T$-process.

We have for any $n$
\begin{equation}\label{d:nabs}
 \int_0^\infty P^t(x,B) e^{-t}dt\ge  \int_{0}^{\infty}\mathbb{P}_x\big(\phi_{t-\tau_n}(\Psi_n)\in B, \tau_n\le t< \tau_{n+1}\big)e^{-t}dt.
\end{equation}
We will show that
we can pick an $n$ such that the measure in the right-hand side of \eqref{d:nabs} has a lower bound as in \eqref{d:lower}. To this end we apply
\cite[Lemma 6.3]{benaim12} to the $(d+1)$-dimensional component of $\SP$.

Assume first that $d=2$ and take $n=2$ in \eqref{d:nabs}. It follows from \eqref{d:Phi} and \eqref{d:Psin} that
\[
\phi_{t-\tau_2}(\Psi_2)=\phi_{t-(\sigma_2+\sigma_1)}(\Psi_2),\quad \Psi_2=\mathbf{S}\big(\phi_{\sigma_2}(\Psi_1),\vartheta_2\big),\quad \Psi_1=\mathbf{S}\big(\phi_{\sigma_1}(x),\vartheta_1\big),
\]
where $\vartheta_k$ are random variables with distribution $\nu(\phi_{\sigma_k}(\Psi_{k-1}),\cdot)$, $k=1,2$, while $\mathbf{S}$ and $\nu$ are as in \eqref{d:S} and \eqref{d:nu}. Let $\sigma$ be an exponentially distributed  random variable independent of all other random variables. Then the right-hand side of  \eqref{d:nabs} is equal to \begin{equation}\label{d:banaim}
\mathbb{P}_x\big(\phi_{\sigma-(\sigma_1+\sigma_2)}(\Psi_2)\in B, \sigma_1+\sigma_2\le \sigma< \sigma_1+\sigma_2+\sigma_3\big).
\end{equation}
Let $x_0=(\xi_0,\zeta_0,i_0)$ with $\xi_0\in (0,1)\times (0,1]$,  $\zeta_0\in [0,\zeta_m^{i_0})$ and $i_0\in \{0,1\}$.
We take two fire occurrences in a single season and the third jump to be the exit time form the given season.
We define $i=i_0$,   $\xi_1=S^{i}_{\theta_1}(\xi_0)$ and $\xi_2=S^{i}_{\theta_2}(\xi_1)$, where $\theta_1\in (0,1)^2$ and $\theta_2\in (0,1)^2$ are such that $p_{\theta_1}^i(\xi_0,\zeta_0)>0$ and $p_{\theta_2}^i(\xi_1,\zeta_0)>0$.
We can always choose such $\theta_1$ and $\theta_2$ by \eqref{e:ptheta}. Recall that the functions $p^i$ are continuous and the jump rate function  $q$, given by $q(x)=\lambda^i(\xi,\zeta)$, is also continuous. This, together with \eqref{d:tail} and \eqref{d:nu}, implies that there is a neighbourhood of $x_0$ such that the distribution of the random variable $(\sigma_1,\sigma_2,\sigma)$ has an absolutely continuous part with respect to the 3-dimensional Lebesgue measure  and with density being bounded below by a positive constant in  a neighbourhood of $(0,0,0)$.
Let us introduce on $\Delta_t=\{(t_1,t_2): t_1,t_2>0, t_1+t_2<t\}$ the following mapping
 \[
 \psi_{(t,\xi, {\boldsymbol\theta})}^i(\mathbf{t})=\varphi^i_{t-(t_1+t_2)}\circ S_{\theta_2}^i\circ \varphi^i_{t_2}\circ  S_{\theta_1}^i\circ \varphi^i_{t_{1}}(\xi)\quad \text{for } \mathbf{t}=(t_1,t_2)\in \Delta_t,  \]
where $t>0$, $\boldsymbol{\theta}=(\theta_1,\theta_2)\in (0,1)^2\times(0,1)^2$, $\xi=(w,g)\in (0,1)\times(0,1]$.
To estimate \eqref{d:banaim} from below it is enough by \cite[Lemma 6.3]{benaim12} to show that the mapping
\[
(\boldsymbol{t},t)\mapsto \big(\psi_{(t,\xi, {\boldsymbol\theta})}^i(\mathbf{t}),\zeta+t\big)
\]
has the derivative of full rank $3$ for small $t$ in a neighbourhood of  $(\xi_0,\zeta_0)$.

Observe that

\begin{equation}
\lim_{\xi\to \xi_0,\, t\to 0}\frac{d\psi_{(t,\xi, {\boldsymbol\theta})}^i(\mathbf{t}) }{d\mathbf{t}}=A,
\end{equation}

where $A$ is the matrix with columns $v_1,v_2$
given by
\[
v_1=DS_{\theta_2}^i(\xi_1)DS_{\theta_1}^i(\xi_0)b^i(\xi_0)-b^i(\xi_2),\quad v_2=DS_{\theta_2}^i(\xi_1)b^i(\xi_1)-b^i(\xi_2),
\]
$D$ denotes the derivative with respect to $\xi$ and $b^i$ is as in \eqref{twoeq}. Now we show that the vectors $v_1$ and $v_2$ are linearly independent.  The transformation $S_\theta^i$  is linear, thus $DS_{\theta}^i=S_{\theta}^i$. Let $S_1=S_{\theta_1}^i$,  $S_2=S_{\theta_2}^i$, and,  to simplify calculations, let $S_j(w,g)=(\alpha_j w,\beta_j g)$, where  $(1-\alpha_j,1-\beta_j)=\theta_j$ by \eqref{d:Stheta}. Then  we have
\[
A=\begin{pmatrix}
     \alpha_2\alpha_1(\alpha_2\alpha_1-1)r_w^i w^2 & \alpha_2(\alpha_2-1)\alpha_1^2r_w^i w^2 \\
     \beta_2\beta_1 r_g^ig\big[(\alpha_2\alpha_1-1)  w+(\beta_2\beta_1-1) g\big] &\beta_2\beta_1r_g^ig\big[(\alpha_2-1)\alpha_1w+(\beta_2-1)\beta_1g\big]
\end{pmatrix}.
\]
We see that
$
\det A=0
$ if and only if
\begin{equation}\label{e:detA2}
\frac{\alpha_1}{\beta_1}\frac{1-\alpha_2}{1-\beta_2}=\frac{1-\alpha_1\alpha_2}{1-\beta_1\beta_2}.
\end{equation}
We conclude that
\[
\det \left[\frac{d\psi_{(t,\xi, {\boldsymbol\theta})}^i(\mathbf{t}) }{d\mathbf{t}}\right]\neq 0
\]
for $\xi$ close to $\xi_0$, sufficiently small $t$ and suitably chosen $\boldsymbol{\theta}$.

Now for the case of $d=4$ we take $n=5$ (two fire occurrences in each season and a switch between the seasons) in \eqref{d:nabs}. Let $\Delta_t=\{(t_1,t_2,t_3,t_4): t_1,t_2,t_3,t_4>0,\linebreak t_1+t_2+t_3+t_4<t\}$ and
\begin{multline*}
 \psi_{(t,\xi,\zeta, {\boldsymbol\theta})}^{i}(\mathbf{t})=\varphi^{1-i}_{t-(t_3+t_4+\zeta_m^i-\zeta)}\circ S_{\theta_4}^{1-i}\circ \varphi^{1-i}_{t_4}\circ  S_{\theta_3}^{1-i}\circ \varphi^{1-i}_{t_{3}}\\ \circ  \varphi^i_{\zeta_m^i-\zeta-(t_1+t_2)}\circ S_{\theta_2}^i\circ \varphi^i_{t_2}\circ  S_{\theta_1}^i\circ \varphi^i_{t_{1}}(\xi)
\end{multline*}
for $\mathbf{t}=(t_1,t_2,t_3,t_4)\in \Delta_t$, $t>0$, $\boldsymbol{\theta}=(\theta_1,\theta_2,\theta_3,\theta_4)$ with each $\theta_j\in (0,1)^2$, and $\xi=(w,g,h_G,h_B)\in (0,1)\times(0,1]\times(0,\infty)^2$.
We take  arbitrary $x_0=(\xi_0,\zeta_0,i_0)$ with $\xi_0\in (0,1)\times (0,1]\times (0,\infty)^2$,  $\zeta_0\in [0,\zeta_m^{i_0})$ and $i_0\in \{0,1\}$. We define $i=i_0$,
\[
\xi_1=\varphi^i_{\zeta_m^i-\zeta_0}(\xi_0), \quad \xi_2=S_{\theta_1}^i(\xi_1),\quad \xi_3=S_{\theta_2}^i(\xi_2),\quad \xi_4=S_{\theta_3}^{1-i}(\xi_3),\quad \xi_5=S_{\theta_4}^{1-i}(\xi_4),
\]
where $\theta_1,\theta_2,\theta_3, \theta_4\in (0,1)^2$ are such that $p_{\theta_j}^i(\xi_j,\zeta_m^i-\zeta_0)>0$ for $j=1,2$ and $p_{\theta_j}^{1-i}(\xi_j,0)>0$ for $j=3,4$.
Similarly as for $d=2$ by \cite[Lemma 6.3]{benaim12} it is enough to show that the mapping
\[
(\boldsymbol{t},t)\mapsto \left(\psi_{(t,\xi,\zeta, {\boldsymbol\theta})}^i(\mathbf{t}),t-(\zeta_m^i-\zeta)\right)
\]
has the derivative of full rank $5$ for a short time of staying in the season $1-i$, i.e. as $t\downarrow \zeta_m^i-\zeta_0$, and in a neighbourhood of  $(\xi_0,\zeta_0)$.
It is easily seen that
\begin{equation}
\lim_{\substack{\xi\to \xi_0,\, \zeta\to\zeta_0,\\ t, t_1\to \zeta_m^i-\zeta_0,\, t_2\to 0 }}\frac{d\psi_{(t,\xi,\zeta, {\boldsymbol\theta})}^i(\mathbf{t}) }{d\mathbf{t}}=A,
\end{equation}
where now $A$ is the matrix with columns $v_1,v_2, v_3,v_4$
given by
\begin{align*}
v_1&=DS_{\theta_4}^{1-i}(\xi_4)DS_{\theta_3}^{1-i}(\xi_3)(DS_{\theta_2}^i(\xi_2)DS_{\theta_1}^i(\xi_1)b^i(\xi_1)-b^i(\xi_3)),\\ v_2&=DS_{\theta_4}^{1-i}(\xi_4)DS_{\theta_3}^{1-i}(\xi_3)(DS_{\theta_2}^i(\xi_2)b^i(\xi_2)-b^i(\xi_3)),\\
v_3&=DS_{\theta_4}^{1-i}(\xi_4)DS_{\theta_3}^{1-i}(\xi_3)b^{1-i}(\xi_3)-b^{1-i}(\xi_5),\\
v_4&=DS_{\theta_4}^{1-i}(\xi_4)b^{1-i}(\xi_4)-b^{1-i}(\xi_5).
\end{align*}
By using the formula for $b(\xi)$ given by the right-hand side of equation  \eqref{no-fire.herb} with $\xi=(w,g,h_G,h_B)$ and by taking $S(\xi)=(\alpha w,\beta g, h_G,h_B)$ for the corresponding  $S_\theta^i$ as in \eqref{d:Stheta}, we obtain
\begin{equation}\label{d:Sb}
S\big(b(\xi)\big)-b\big(S(\xi)\big)=\begin{pmatrix}
\alpha(\alpha-1) r_w w^2\\
\beta r_gg\big[(\alpha -1) w  +(\beta-1)g\big]\\
(1-\beta)e_gh_G g\\
(1-\alpha)e_w h_B w
\end{pmatrix} \quad\text{for }\xi=(w,g,h_G,h_B).
\end{equation}
Let us take $S_j=S_{\theta_j}^i$ for $j=1,2$ and $S_j=S_{\theta_j}^{1-i}$ for $j=3,4$ so that $S_j(w,g,h_G,h_B)=(\alpha_j w,\beta_j g, h_G,h_B)$ with $(1-\alpha_j,1-\beta_j)=\theta_j$.
Applying \eqref{d:Sb} with
$r_w=r_w^{i}$ and $r_g=r_g^{i}$ and appropriate $\alpha,\beta$, the vector $v_1$ with  $\xi_1=(w,g,h_G,h_B)$ is of the form
\[
v_1=\begin{pmatrix}
\alpha_4\alpha_3\alpha_2\alpha_1(\alpha_2\alpha_1-1)r_w^i w^2 \\
\beta_4\beta_3\beta_2\beta_1r_g^ig\big[(\alpha_2\alpha_1 -1) w  +(\beta_2\beta_1-1)g\big] \\
(1-\beta_2\beta_1)e_gh_Gg\\
(1-\alpha_2\alpha_1)e_w h_B w
\end{pmatrix}.
\]
Similarly, we obtain
\[
v_2=\begin{pmatrix}
\alpha_4\alpha_3\alpha_2\alpha_1^2(\alpha_2-1)r_w^i w^2\\
\beta_4\beta_3\beta_2\beta_1r_g^ig\big[(\alpha_2 -1)\alpha_1 w  +(\beta_2-1)\beta_1g\big]\\
(1-\beta_2)\beta_1e_gh_Gg\\
(1-\alpha_2)\alpha_1 e_w h_B w
\end{pmatrix}.
\]
Next observe that
\[
v_3=\begin{pmatrix}
\alpha_4\alpha_3 (\alpha_4\alpha_3-1)\alpha_2^2\alpha_1^2 r_w^{1-i}w^2 \\
\beta_4\beta_3\beta_2\beta_1 r_g^{1-i}g\big[(\alpha_4\alpha_3-1)\alpha_2\alpha_1w +(\beta_4\beta_3-1)\beta_2\beta_1 g\big] \\
(1-\beta_4\beta_3)\beta_2\beta_1e_gh_Gg\\
(1-\alpha_4\alpha_3)\alpha_2\alpha_1 e_w h_B w
\end{pmatrix}
\]
and
\[
v_4=\begin{pmatrix}
\alpha_4 (\alpha_4-1)\alpha_3^2\alpha_2^2\alpha_1^2 r_w^{1-i}w^2\\
 \beta_4\beta_3\beta_2\beta_1 r_g^{1-i}g\big[(\alpha_4-1)\alpha_3\alpha_2\alpha_1w +(\beta_4-1)\beta_3\beta_2\beta_1 g\big] \\
(1-\beta_4)\beta_3\beta_2\beta_1e_gh_Gg\\
(1-\alpha_4)\alpha_3\alpha_2\alpha_1 e_w h_B w
\end{pmatrix}.
\]
Using Gaussian elimination it is easily seen that the first two coordinates of $v_1$ and $v_2$ can be made zero and
hence $\det A=0$ if and only if \eqref{e:detA2} holds or
\begin{equation}\label{e:detA4}
\frac{\alpha_3}{\beta_3}\frac{1-\alpha_4}{1-\beta_4}=\frac{1-\alpha_3\alpha_4}{1-\beta_3\beta_4}.
\end{equation}
Consequently, we can find $\theta_j=(1-\alpha_j,1-\beta_j)$, $j=1,2,3,4,$ such that both \eqref{e:detA2} and \eqref{e:detA4} do not hold implying that
\[
\det\frac{d\psi_{(t,\xi,\zeta, {\boldsymbol\theta})}^i(\mathbf{t}) }{d\mathbf{t}}\neq 0
\]
for $t$ close to $\zeta_m^{i_0}-\zeta_0$ and $(\xi,\zeta)$ in a neighbourhood of $(\xi_0,\zeta_0)$.
\end{proof}

\section{Proof of Theorem~\ref{t:mean}}\label{s:proof}

The resolvent kernel $R\colon \SP\times \mathcal{B}\to [0,1]$ is defined as
\[
R(x,B)=\int_{0}^\infty e^{-t}P^t(x,B)dt.
\]
The kernel $R$ is the transition probability for the discrete-time Markov chain $\breve{\Phi}$ that is defined by observing the process $\Phi$ at jump-times of a Poisson process with intensity~$1$ that is independent of the process $\Phi$. We call this chain the \emph{$R$-chain}.
We say that the $R$-chain is a \emph{$T$-chain} if there is a probability distribution $b=(b_k)$ on $\mathbb{Z}_+$ and a non-trivial continuous component for the kernel
\[
R_b(x,B)=\sum_{n=0}^\infty b_n R^n(x,B).
\]

Following \cite{meyn_tweedieI} and \cite{meyn_tweedieII} we say that a trajectory converges to infinity
if it visits each compact set only finitely many times and
we write $\{\breve{\Phi}\to\infty\}$ for  the $R$-chain and  $\{\Phi\to \infty\}$ for the process $\Phi$.

\begin{lemma}\label{p:tch}
If the $R$-chain $\breve{\Phi}$ is a $T$-chain then $\Phi$ is a $T$-process and
\begin{equation}\label{e:evan}
\mathbb{P}_x\{\breve{\Phi}\to\infty\}=\mathbb{P}_x\{\Phi\to\infty\},\quad x\in \SP.
\end{equation}
If $\mathbb{P}_x\{\Phi\to\infty\}<1$ for all $x\in \SP$ and $\Phi$ is a $T$-process then the $R$-chain is a $T$-chain.
\end{lemma}
\begin{proof}
Since the $n$th jump of the Poisson process has the Erlang distribution,  we have
\[
R^n(x,B)=\int_0^\infty e^{-t}\frac{t^{n-1}}{(n-1)!}P^t(x,B)dt.
\]
If we consider the probability measure
\[
a(dt)=\sum_{n=0}^\infty b_n e^{-t}\frac{t^n}{n!}dt
\]
on $\mathbb{R}_+$, where $b=(b_n)$ is a probability measure on $\mathbb{Z}_+$, then the kernel $K_a$ has the same continuous component as $R_b$. The equality in \eqref{e:evan} follows from \cite[Proposition 3.2]{meyn_tweedieII}. The converse statement is  \cite[Theorem 4.1 (iii)]{meyn_tweedieII}.
\end{proof}

The $R$-chain is called a \emph{mean ergodic} chain on $\SP$ if for each probability measure $\mu\in \mathcal{M}(\SP)$ there exists a measure $\mu\Pi\in \mathcal{M}(\SP)$ such that
\begin{equation}\label{d:meg}
\lim_{n\to\infty}\left\|\frac{1}{n}\sum_{k=0}^{n-1} \mu R^k -\mu\Pi\right\|=0.
\end{equation}
Observe that the measure $\pi=\mu\Pi$ in condition \eqref{d:meg} is invariant for the $R$-chain, i.e. $\pi R=\pi$.  It is known (see \cite{azema67}) that a measure $\pi$ is invariant for the process $\Phi$ if and only if it is invariant for the $R$-chain. We now show that the convergence in \eqref{d:meg} is equivalent to the one in \eqref{c:meg}.

\begin{lemma}\label{t:meane}
The process $\Phi$ is mean ergodic if and only if the $R$-chain is mean ergodic on $\SP$. Moreover, for any bounded Borel measurable $f$ we have
\begin{equation*}
\lim_{t\to\infty}\frac{1}{t}\int_0^t f\big(\Phi(s)\big)ds=\lim_{n\to\infty}\frac{1}{n}\sum_{k=1}^n f(\breve{\Phi}_k),
\end{equation*}
if any of the pointwise limits exist.
\end{lemma}
\begin{proof}
For any probability measure $\mu$ on $\B$ we define the resolvent operator of $P^t$ by
\[
\mu U_\alpha (B)=\int_0^\infty e^{-\alpha t}\mu P^t(B)dt, \quad \alpha>0,B\in \B.
\]
We have $\mu U_1=\mu R$ and
\begin{equation}
\mu U_\alpha (B)=\sum_{k=1}^\infty (1-\alpha)^{k-1}\mu R^k(B),\quad B\in \B.
\end{equation}
First observe that the Ces\'aro convergence in \eqref{d:meg} implies the Abel convergence \[
\lim_{\alpha\to 0^{+}}\alpha \sum_{k=1}^\infty (1-\alpha)^{k-1}\mu R^k=\mu \Pi,
\]
see e.g. \cite[Theorem 2.1]{krengel85}, and leads to
\begin{equation}\label{c:aeg}
\lim_{\alpha\to 0^{+}}\|\alpha \mu U_\alpha-\mu \Pi\|=0.
\end{equation}
Condition \eqref{c:aeg} implies \eqref{d:meg} by \cite[Theorem 3.1]{emilion} and
\eqref{c:meg} by \cite[Theorem 3.3]{emilion}. Finally, the implication leading from \eqref{c:meg} to \eqref{c:aeg} follows by using standard arguments. The second part follows from \cite[Theorem 5.1.1]{bremaud92}.
\end{proof}

We need to introduce more notation.
The following notions will be presented  only for the  continuous time process $\Phi$, but  analogous definitions are valid for the discrete time $R$-chain  $\check{\Phi}=\{\check{\Phi}_k\}$. We refer to \cite{meyn09} for the general theory of discrete time Markov chains.

Given a measurable set $B$ we define  the first hitting time of the set $B$ and the number of visits to $B$ respectively  by
\[
\tau_B=\inf\{t> 0: \Phi(t)\in B\}\quad \text{and}\quad \eta_B=\int_{0}^\infty \mathbf{1}\{\Phi(t)\in B\}dt.
\]
A set $B$ is called
(stochastically) \emph{closed} for the process if  $B\neq\emptyset$ and  $\mathbb{P}_x\{\Phi(t)\in B \text{ for all }t\ge 0\}=1$ for $x\in B$. A closed set $B$ is said to be \emph{maximal} if $x\in B \Longleftrightarrow \mathbb{P}_x\{\eta_B=\infty\}=1$.
A set $H$ is
called a \emph{Harris set} for the process $\Phi$ if it is closed and if
there exists some $\sigma$-finite measure $\psi$ such that  $\mathbb{P}_x\{\eta_B=\infty\}=1$ for all $x\in H$ and all $B\in \B$ with $\psi(B)>0$. A set $H$ is
called a \emph{maximal Harris set} if it is a Harris set and a maximal closed set. The process restricted to a maximal Harris set $H$ has an essentially unique invariant measure  on $H$. If the measure is finite then it can be normalized and the process has a unique invariant probability measure on $H$. In that case the set $H$ is called a \emph{positive Harris set}.

\begin{lemma}\label{t:dec}
Suppose that condition (V) holds. Then $\mathbb{P}_x\{\Phi\to\infty\}=0$ for all $x\in \SP$.
If the process $\Phi$ is a $T$-process then the space $\SP$ has the decomposition into disjoint sets
\[
\SP=\bigcup_{i=1}^N H_i\cup E=H\cup E,
\]
where each  $H_i$ is a positive Harris set   and $\mathbb{P}_x\{\eta_H=\infty\}=1$ for all $x\in \SP$. Moreover, the $R$-chain is mean ergodic on $\SP$.
\end{lemma}
\begin{proof}
The function $V$ in condition (V) is norm-like and satisfies $\mathcal{L}V(x)\le d\mathbf{1}_{C}(x)$ for all $x\in \SP$. Thus condition (CD1) of \cite{meyn_tweedieII} holds and \cite[Theorem 3.1]{meyn_tweedieII} implies that $\mathbb{P}_x\{\Phi\to\infty\}=0$ for all $x\in \SP$.  The Doeblin decomposition  \cite[Theorem 4.1]{meyn_tweedieII} and  \cite[Theorem 4.6]{meyn_tweedieIII} show that the space $\SP$ has the required decomposition. It follows from \cite[Theorem 2.1]{meyn_tweedieII} that
\[
\mathbb{P}_x\{\check{\tau}_H<\infty\}=\mathbb{E}_x\big(1-\exp(-\eta_H)\big),\quad x\in \SP,
\]
where $\check{\tau}_H=\inf\{k\ge 1: \check{\Phi}_k\in H\}$ is
the first hitting time of $H$ by the $R$-chain. Consequently, $\mathbb{P}_x\{\check{\tau}_H<\infty\}=1$ for all $x\in \SP$.

 From \cite[Theorem 2.1]{tweedie79} extended in \cite{costa2006} to the case of Borel right process it follows that a set is a maximal Harris set for the process $\Phi$ if and only if its is a maximal Harris set for the $R$-chain.
Hence, the $R$-chain restricted to the set $H_i$ is a positive  Harris recurrent chain with the unique invariant probability measure $\pi_i$. By \cite[Theorem 1.2]{lasserre01} for each $x\in H_i$ we have
\[
\lim_{n\to \infty}\frac{1}{n}\sum_{k=0}^{n-1} R^k(x,\cdot)=\pi_i,
\]
where the convergence is in the total variation norm on  $\mathcal{M}(H_i)$. Thus the $R$-chain is mean ergodic on each set $H_i$. The rest of the proof is similar to the proof of part (i) of \cite[Theorem 7.1]{meyn_tweedieI}.
\end{proof}

\begin{remark}\label{r:pi}
It should be noted that the limiting measure $\mu\Pi$ in \eqref{d:meg} is of the form
\[
\mu\Pi(B)=\int_{\SP}\Pi(x,B)\mu(dx),
\]
where the kernel $\Pi$ is given by  \cite[Theorem 7.1]{meyn_tweedieI}
\begin{equation}\label{d:Pi}
\Pi(x,B)=\sum_{i=1}^N \pi_i(B\cap H_i)\mathbb{P}_x\{\check{\tau}_{H_i}<\infty\}, \quad x\in \SP, B\in \mathcal{B},
\end{equation}
and $\pi_i$, $i=1,\ldots N$, are invariant probability measures.
Moreover, as in the proof of   \cite[Theorem 7.1]{meyn_tweedieI} we obtain that for any bounded Borel measurable~$f$
\[
\mathbb{P}_x\left(\lim_{n\to\infty}\frac{1}{n}\sum_{k=1}^n f(\breve{\Phi}_k)=\int fd\tilde{\Pi}\
\right)=1, \quad x\in \SP,
\]
where the random measure $\tilde{\Pi}$ is defined as
\[
\tilde{\Pi}(B) =\sum_{i=1}^N \mathbf{1}(\breve{\tau}_{H_i}<
\infty)\pi_i(B\cap H_i).
\]
\end{remark}

Theorem~\ref{t:mean} is a direct consequence of Lemmas \ref{t:meane} and \ref{t:dec} together with Remark~\ref{r:pi}.

\section{Discussion}

In the present paper we propose a novel approach to the study of seasonal dynamics. It can be applied to stochastic models in population dynamics that underlie periodic changes to its parameters. Especially we provide sufficient conditions for coexistence of competing species.
As a model we introduce two PDMPs describing behavior in each season as the system switching between them in given constant periods of time (season lengths). This may be generalized to more seasons than two. Such description needs additional time variable to keep track of the duration of stay in the present season, leading to time-homogeneous Markov processes. Therefore one cannot use the usual approach to study convergence of distributions. We explore the time averages instead and provide sufficient conditions for their convergence.

The common way to study the effects of seasonality on the dynamics of populations modeled with differential equations is to consider periodically forced  parameters \cite{cushing77,cushing80}. Such models are very difficult to treat analytically although there exist general tools for a study of non-autonomous differential equations with continuous and periodic functions of time \cite{holmes83,kloeden2011}.
A frequently used numerical approach is  bifurcation analysis, first used in this context in \cite{kuznetsov92, rinaldi93}, where for simplicity,  the forcing is
of the form \[
c(t)=c_0(1+\varepsilon \sin(2\pi t)),
\]
with $c_0$ being any model parameter and $\varepsilon$ denoting the forcing amplitude (see   \cite{taylor2013} and the references therein).

Another attempt to model seasonal effects is related to the so-called (seasonal) succession dynamics \cite{klausmeier2010} or, formally similar, behavior shift \cite{tyson2016},
in which the model equations change between seasons. A detailed analysis is possible in simple models \cite{hsuzhao2012}. By changing growth parameters in our equations \eqref{no-fire} and  \eqref{no-fire.herb}  to piecewise-constant periodic functions of time we get examples of this dynamics, with a particular behavior illustrated as in Figure~\ref{fig:herb}. This approach, in contrary to the situation in the previous paragraph, gives a discontinuous periodic forcing and can simplify the analysis.
Including seasonality might still not  support coexistence of species, as in the case of  model \eqref{no-fire},  since positive solutions of both systems converge to the same equilibrium $(1,0)$ representing woodland. Modeling fire impact on vegetation introduces stochasticity into our systems and can have a positive effect on survival of all species. Especially, adding fire alone or together with herbivores  prevents an overgrowth of trees and allows existence of mixed woodland-grassland ecosystem reflecting savanna.

In general, savanna models incorporate fire disturbances into model equations in a deterministic way \cite{staver2011tree,yatat2018tribute,hoyer-leitzel2021}. To our knowledge there exists only a discrete-time matrix model \cite{accatino2013humid} that contains both, seasonality and fire-vegetation feedback, but it does not provide any analytical insight focusing mainly on simulations. We propose the analytically tractable continuous-time models, although they are less convenient to simulate and limited to discrete losses of the biomass while it would be more realistic to model impact of fire in a spatially explicit way.



%
We were not studying sufficient conditions for the  uniqueness of invariant distributions in our models and leave it to a future work. Once uniqueness is obtained then the law of large numbers from Theorem \ref{th:main} implies automatically  stochastic persistence \cite{benaim2018, benaim2016, nguyen2018,  guillin2019arxiv} of considered  populations. It would be also interesting to study extinction \cite{benaim2016,nguyen2018, nguyen2020}.
Our approach can be used to extend other stochastic models like \cite{nguyen2018} by adding seasonal effects.

\bibliographystyle{siamplain}

\bibliography{references_siam}

\end{document}